\documentclass[11pt, a4paper]{article}

\usepackage[utf8]{inputenc}
\usepackage{amsmath, amsthm, amsfonts, amssymb}

\usepackage{nicefrac, xfrac}

\usepackage[bookmarks]{hyperref}
\usepackage{indentfirst} 

\usepackage{authblk}
\usepackage{lipsum}

\newtheorem{lemma}{Lemma}[section]
\newtheorem{theorem}[lemma]{Theorem}
\newtheorem{proposition}[lemma]{Proposition}

\newtheorem{result}{Theorem}

\theoremstyle{definition}

\newcommand{\abs}[1]{\ensuremath{\left| #1 \right|}}
\newcommand{\op}{\operatorname}
\newcommand{\ce}[2]{\operatorname{C}_{#1}(#2)}
\newcommand{\no}[2]{\operatorname{N}_{#1}(#2)}
\newcommand{\ze}[1]{\operatorname{Z}(#1)}

\newcommand{\rad}[2]{\op{O}_{#1}(#2)}

\newcommand{\rup}[2]{\op{O}^{#1}(#2)}
\newcommand{\syl}[2]{\op{Syl}_{#1}\left(#2\right)}

\newcommand{\groupgen}[1]{\langle #1 \rangle}

\newcommand{\irr}{\operatorname{Irr}}

\newcommand{\gal}{\operatorname{Gal}}

\newcommand{\smid}{\! \mid \!}

\DeclareMathSymbol{\shortminus}{\mathbin}{AMSa}{"39}

\renewcommand{\phi}{\varphi}
\renewcommand{\epsilon}{\varepsilon}

\title{On the degrees of irreducible characters fixed by some field automorphism, $p$-solvable groups}

\date{\today}

\author{Nicola Grittini}

\begin{document}

\maketitle

\begin{abstract}
It is known that, if all the real-valued irreducible characters of a finite group have odd degree, then the group has normal Sylow $2$-subgroup.

We generalize this result for Sylow $p$-subgroups, for any prime number $p$, while assuming the group to be $p$-solvable. In particular, it is proved that a $p$-solvable group has a normal Sylow $p$-subgroup if $p$ does not divide the degree of any irreducible character of the group fixed by a field automorphism of order $p$.
\end{abstract}

\section{Introduction}

The search for results linking the normal structure of a finite group and the degrees of its irreducible characters is a classical problem in character theory. The best known result linking the two properties is the celebrated Ito-Michler Theorem, and many variants of this theorem have already been found.

In \cite{Dolfi-Navarro-Tiep:Primes_dividing} a variant of Ito-Michler Theorem is studied, involving real-valued characters. In fact, \cite[Theorem A]{Dolfi-Navarro-Tiep:Primes_dividing} states that, if the prime 2 does not divide the degree of any real-valued irreducible character of a finite group $G$, then $G$ has a normal Sylow 2-subgroup.

If we look at the real-valued characters as the characters fixed by the complex conjugation, it is possible to generalize this theorem for any prime $p$, at least for $p$-solvable groups.

\begin{result}
\label{result:p-solvable_groups}
Let $G$ be a finite $p$-solvable group, for some prime $p$, and let $\sigma \in \gal(\mathbb{Q}_{\abs{G}} / \mathbb{Q})$ of order $p$. Suppose that $p$ does not divide the degree of any $\sigma$-invariant irreducible character of $G$. Then, $G$ has a normal Sylow $p$-subgroup.
\end{result}

Theorem~\ref{result:p-solvable_groups} generalizes \cite[Theorem A]{Dolfi-Navarro-Tiep:Primes_dividing} for $p$-solvable groups, since \cite[Theorem A]{Dolfi-Navarro-Tiep:Primes_dividing} follows directly from Theorem~\ref{result:p-solvable_groups} if we take the complex conjugation as $\sigma$. In fact, the proof of Theorem~\ref{result:p-solvable_groups} is sometimes similar to the proof of \cite[Theorem A]{Dolfi-Navarro-Tiep:Primes_dividing}; however, we also use some original techniques.

If we assume the group $G$ not to be $p$-solvable, Theorem~\ref{result:p-solvable_groups} is no longer valid, as we can see if we consider $\op{PSL}(2,8)$ or $\op{PSL}(2,32)$ and $p=3$, or $\op{Sz}(32)$ and $p=5$. However, in this situation, a weaker result can still be found.

\begin{result}
\label{result:general_case}
Let $G$ be a finite group, for some prime $p$, let $\sigma \in \gal(\mathbb{Q}_{\abs{G}} / \mathbb{Q})$ of order $p$ and let $P$ be a Sylow $p$-subgroup of $G$. Suppose that $p$ does not divide the degree of any $\sigma$-invariant irreducible character of $G$. Then, $\no{G}{P}$ intersects every chief factor of $G$ non-trivially.
\end{result}

Despite our restriction to the $p$-solvable case, the proof of Theorem~\ref{result:p-solvable_groups} requires some non-trivial research on non-abelian simple groups. In particular, we need to answer to the following question. Suppose that $S$ is a non-abelian simple group, let $p$ be a prime number, let $S < G \leq \op{Aut}(S)$ such that $G/S$ is a $p$-group (here, we conveniently identify $S$ with $\op{Inn}(S)$) and let $\sigma \in \gal(\mathbb{Q}_{\abs{G}} / \mathbb{Q})$ be of order $p$. Is it true that $G$ has a $\sigma$-invariant irreducible character of degree divisible by $p$?

In order to prove Theorem~\ref{result:p-solvable_groups}, we only need to answer to this question when $p \nmid \abs{S}$. In particular, we can assume $p$ to be odd and $S$ to be a simple group of Lie type. In this setting, however, the result we are able to find is stronger then what we need.

\begin{result}
\label{result:simple_groups}
Let $S$ be a non-abelian simple group of Lie type, let $p$ be an odd prime number, let $S < G \leq \op{Aut}(S)$ such that $G/S$ is a $p$-group and let $\sigma \in \gal(\mathbb{Q}_{\abs{G}} / \mathbb{Q})$ be of order $p$. In the following situations, $G$ has a $\sigma$-invariant irreducible character of degree divisible by $p$.
\begin{itemize}
\item[a)] $S$ is of type $A_n(q)$, and $p \nmid \abs{\tilde{S} : \tilde{S} \cap G}$ for some $S < \tilde{S} \leq \op{Aut}(S)$ isomorphic to $\op{PGL}_{n+1}(q)$ and such that $G \nleq \tilde{S}$;
\item[b)] $S$ is of type ${}^2A_n(q^2)$, and $p \nmid \abs{\tilde{S} : \tilde{S} \cap G}$ for some $S < \tilde{S} \leq \op{Aut}(S)$ isomorphic to $\op{PGU}_{n+1}(q)$ and such that $G \nleq \tilde{S}$;
\item[c)] $S$ is of type $D_n(q)$, and $p \neq 3$ if $n=4$;
\item[d)] $S$ is of type $E_6(q)$ or ${}^2E_6(q^2)$, and $p \neq 3$;
\item[e)] $S$ is of type $B_n(q)$, $C_n(q)$, $E_7(q)$, $E_8(q)$, $F_4(q)$, $G_2(q)$, ${}^2D_n(q^2)$, ${}^3D_4(q^3)$, ${}^2B_2(2^{2n+1})$, ${}^2F_4(2^{2n+1})$, ${}^2G_2(3^{2n+1})$.
\end{itemize}
In particular, the above conditions are always verified if $p \nmid \abs{S}$.
\end{result}

The list of Theorem~\ref{result:simple_groups} is by no means exhaustive, it only includes the cases we can prove by using the techniques presented in this paper. For instance, a $\sigma$-invariant irreducible character of degree divisible by $p$ exists for sure also when $S$ is defined over a field of characteristic $p$, since the Steinberg character extends to $\op{Aut}(S)$. Moreover, the results in \cite{Dolfi-Navarro-Tiep:Primes_dividing} \textit{almost} provide a proof of the existence of such character (for any $\sigma$) when $p=2$. Indeed, it is possible that a $\sigma$-invariant irreducible character of degree divisible by $p$ exists for any choice of $p$, $G$, $S$ and $\sigma$, provided that $p \mid \abs{G:S}$.

\section{Preliminary results}

We shall prove some preliminary results first. We begin with some results about $p$-groups acting on vector spaces. We will then apply them to the action on characters of elementary abelian groups. 

\begin{lemma}
\label{lemma:action_on_minimal}
Let $p$ be a prime number and let $P > 1$ be a $p$-group. Let $V$ be a finite vector space which is also an irreducible $P$-module, such that $\ce{V}{P}=0$. Let $\sigma \in \op{Aut}(V)$ of order $p$ such that, for every $v \in V$ and every $x \in P$, $(v^x)^{\sigma} = (v^{\sigma})^x$. Then, there exists $g \in P$ such that $v^{\sigma}=v^g$ for every $v \in V$.
\end{lemma}

\begin{proof}
Notice that, since $\ce{V}{P}=0$, then $\ce{P}{V} < P$ and, thus, if we replace $P$ with $\bar{P}=P/\ce{P}{V}$ we may assume $\ce{P}{V}=1$. Moreover, since $V$ is an irreducible $P$-module, then $V=\groupgen{v^P}$ for every $0 \neq v \in V$.

Let $z \in Z = \ze{P}$ and suppose $z \in \ce{P}{v}$ for some non-trivial $v \in V$; then, for every $y \in P$, $(v^y)^z=(v^z)^y=v^y$. Since by hypothesis $V=\groupgen{v^P}$, it follows that $z$ centralizes all the generators of $V$ and, thus, $z \in \ce{P}{V} = 1$. Therefore, $Z \cap \ce{P}{v} = 1$ for every $0 \neq v \in V$.

Let now $H = \groupgen{\sigma} \times Z$ and consider the group $\Gamma = H \ltimes V$. Since $H$ is an abelian $p$-group and it is not cyclic, it follows from \cite[Corollary~6.10]{Isaacs:Finite_Group_Theory} that $\Gamma$ is not a Frobenius group. As a consequence of \cite[Theorem~6.4]{Isaacs:Finite_Group_Theory}, there exists $w \in V \setminus \{0\}$ such that $D = \ce{H}{w} \neq 1$. However, we know from the previous paragraph that $Z \cap D = 1$; thus, there exists $z \times \sigma^{-m} \in D$ for some $m$ coprime with $p$ and some $z \in Z$. Now, let $a \in \mathbb{Z}$ such that $am=bp + 1$ for some $b \in \mathbb{Z}$, then $(z \times \sigma^{-m})^a = z^a \times \sigma^{-1} \in D$ and, thus, $w^{z^a} = w^{\sigma}$. We may take $g = z^a \in \ze{P}$.

Let $v \in V$. Since, by hypothesis, $V=\groupgen{w^P}$, we can write $v = \sum_{i=1}^k n_i w^{p_i}$ for some $p_i \in P$ and some $n_i \in \mathbb{Z}$. Now we have that 
$$ v^{\sigma} = \sum_{i=1}^k (n_i w^{p_i})^{\sigma} = \sum_{i=1}^k (n_i w^{\sigma})^{p_i} = \sum_{i=1}^k (n_i w^g)^{p_i} = \sum_{i=1}^k (n_i w^{p_i})^g = v^g $$
and the thesis follows.
\end{proof}

\begin{lemma}
\label{lemma:action_on_space}
Let $p$ be a prime number and let $P > 1$ be a $p$-group. Let $V$ be a finite vector space which is also a $P$-module, and suppose that $\ce{V}{P}=0$. Let $\sigma \in \op{Aut}{V}$ of order $p$ such that, for every $v \in V$ and every $x \in P$, $v^{\sigma} \in \groupgen{v}$ and $(v^x)^{\sigma} = (v^{\sigma})^x$. Then, there exists $g \in P$ and $0 \neq v \in V$ such that $v^{\sigma}=v^g$ and $\abs{P:\ce{P}{v}} = \min\{ \abs{P:\ce{P}{w}} \mid 0 \neq w \in V \}$.
\end{lemma}

\begin{proof}
At first, we recall that, since $V$ is a finite vector space, then it is a vector space over a field of finite characteristic $q$. From the fact that $\ce{V}{P}=0$, it follows that $q \neq p$.

Thus, by Maschke's Theorem, we can write $V$ as a sum of irreducible $P$-modules $V = W_1 \oplus ... \oplus W_n$. Let $v \in V$ such that $\abs{P:\ce{P}{v}} = \min\{ \abs{P:\ce{P}{w}} \mid 0 \neq w \in V \}$, we claim that $v$ can be chosen such that $v \in W_i$ for some $i \in \{1,...,n\}$. In fact, let $v = w_1 + ... + w_n$ for some uniquely determined $w_i \in W_i$, for each $i \in \{1,...,n\}$. Then, $\ce{P}{v} = \bigcap_{i=1}^n \ce{P}{w_i}$ and, by the minimality of $\abs{P:\ce{P}{v}}$, it follows that, for each $i \in \{1,...,n\}$, either $w_i=0$ or $\ce{P}{w_i}=\ce{P}{v}$. Since $v \neq 0$, then there exists $i \in \{1,...,n\}$ such that $w_i \neq 0$ and, if we replace $v$ with $w_i$, we prove the claim.

Thus, without loss of generality, we can claim that there exists $v \in W_1$ such that $\abs{P:\ce{P}{v}} = \min\{ \abs{P:\ce{P}{w}} \mid 0 \neq w \in V \}$. Moreover, notice that ${W_1}^{\sigma} \leq W_1$ and, thus, $\sigma_{\mid_{W_1}} \in \op{Aut}(W_1)$ and either $\sigma_{\mid_{W_1}} = 1$ or it has order $p$. Since $\ce{W_1}{P} \leq \ce{V}{P} = 0$ and $W_1$ is an irreducible $P$-module, it follows from Lemma~\ref{lemma:action_on_minimal} that there exists $g \in P$ such that $v^{\sigma}=v^g$.
\end{proof}

We now prove some preliminary results concerning character theory.

\begin{proposition}
\label{proposition:invariant_in_subgroup}
Let $G$ be a finite group, let $p$ be a prime number and let $\sigma \in \gal(\mathbb{Q}_{\abs{G}} / \mathbb{Q})$ such that $o(\sigma)$ is a power of $p$. Suppose $H \leq G$ such that $p \nmid \abs{G:H}$ and suppose, for some $\psi \in \irr(H)$, that $\psi(1)$ divides the degree of every irreducible constituent of $\psi^G$. If $\psi$ is fixed by $\sigma$, then $\sigma$ also fixes some irreducible constituents of $\psi^G$.
\end{proposition}

Notice that the hypothesis of $\psi(1)$ dividing the degree of every irreducible constituent of $\psi^G$ is automatically verified if $H$ is normal in $G$.

\begin{proof}
Since $\psi^{\sigma}=\psi$, then also $(\psi^G)^{\sigma}=\psi^G$ and, thus, $\sigma$ acts on the irreducible constituents of $\psi^G$. Let $C_1,...,C_k$ be the orbits of this action, then we have
$$ \abs{G:H}\psi(1) = \psi^G(1) = \sum_{i=1}^k \abs{C_i} n_i \chi_i(1), $$
where $\chi_i$ is a representative of $C_i$ for each $i \in \{1,...,k\}$, and $n_i$ is the multiplicity of $\chi_i$ in $\psi^G$. Notice that, since $o(\sigma)=p^a$ for some $a \in \mathbb{N}$, then the size of each orbit $\abs{C_i}$ is a power of $p$.

It follows that
$$ \abs{G:H} = \sum_{i=1}^k \abs{C_i} n_i \frac{\chi_i(1)}{\psi(1)} $$
and, since $p \nmid \abs{G:H}$ and the quotients $\frac{\chi_i(1)}{\psi(1)}$ are all integers, we have that $\abs{C_i}=1$ for some $i \in \{1,...,k\}$ and, thus, there exists some $\chi = \chi_i \in \irr(G \smid \psi)$ such that $\chi^{\sigma}=\chi$.
\end{proof}

\begin{lemma}
\label{lemma:normal_coprime}
Let $G$ be a finite group, let $N \lhd G$ and let $\phi \in \irr(N)$ such that $(\abs{I_G(\phi)/N},o(\phi)\phi(1))=1$. Let $\sigma \in \gal(\mathbb{Q}_{\abs{G}} / \mathbb{Q})$ and suppose there exists $g \in G$ such that $\phi^{\sigma} = \phi^g$. Then, there exists $\chi \in \irr(G \smid \phi)$ such that $\chi^{\sigma} = \chi$. Moreover, $\chi$ can be chosen such that $\chi(1) = \phi(1) \abs{G:I_G(\phi)}$.
\end{lemma}

\begin{proof}
Let $T=I_{G}(\phi)$ and notice that $T=I_{G}(\phi^{\sigma})$. Let $n = o(\phi)$, then by \cite[Corollary~6.28]{Isaacs} there exists a unique character $\hat{\phi} \in \irr(T \smid \phi)$ extending $\phi$ such that $o(\hat{\phi})=n$. Moreover, $\hat{\phi}^{\sigma}$ lies over $\phi^{\sigma}$ and, analogously, it is the unique character in $\irr(T \smid \phi^{\sigma})$ such that $o(\hat{\phi}^{\sigma})=n$. Since $\phi^g = \phi^{\sigma}$, then $T^g=T$ and $\hat{\phi}^g$ lies over $\phi^{\sigma}$; moreover, $o(\hat{\phi}^g) = o(\hat{\phi})=n$ and it follows by uniqueness that $\hat{\phi}^g = \hat{\phi}^{\sigma}$.

Let $\chi = \hat{\phi}^G \in \irr(G)$ and notice that $\chi(1) = \abs{G:T}\phi(1)$. Moreover, since $(\hat{\phi}^{\sigma})^G = \chi^{\sigma}$, we have that
$$ \chi = (\hat{\phi}^g)^G = (\hat{\phi}^{\sigma})^G = \chi^{\sigma}. \eqno\qedhere $$
\end{proof}

Finally, we cite a result from \cite{Navarro-Tiep:Rational_irreducible_characters}, which we will use to justify the fact that a non-abelian simple group always has a non-principal $\sigma$-invariant character, for any choice of $\sigma$.

\begin{theorem}[{\cite[Theorem 8.2]{Navarro-Tiep:Rational_irreducible_characters}}]
\label{theorem:rational_char}
Let $G$ be a finite group. If $G$ has exactly one irreducible rational character, then $G$ has odd order.
\end{theorem}

\section{Simple groups}

In this section, we will prove some properties of finite simple groups of Lie type, which then will lead to the proof of Theorem~\ref{result:simple_groups}. We will rely often to the Deligne-Lusztig theory, which connects the irreducible characters of a finite group of Lie type with the characters of the maximal tori of the group. A complete description of the theory can be found in \cite{Geck-Malle:Character_Theory}. Notice that most presentations of the Deligne-Lusztig theory prefer to consider semisimple elements of a dual group in place of the irreducible characters of the maximal tori. In the context of this paper, however, it is probably better to consider characters of subgroups, since it will make our arguments clearer.

Despite the use of the Deligne-Lusztig theory, however, some steps of our proofs still rely on the Classification of finite simple groups.

We first need to prove a result on characters of abelian groups, which then we will apply to the characters in maximal tori.

\begin{lemma}
\label{lemma:torus}
Let $p$ be an odd prime number, let $A$ be an abelian group and let $A_0 \leq A$ be cyclic of order $\ell^f - 1$, for some prime number $\ell$ and some positive integer $f$ such that $p \mid f$. Let $\sigma \in \gal(\mathbb{Q}_{\abs{A}} / \mathbb{Q})$ of order $p$ and let $\nu \in \op{Aut}(A)$ such that, for every $x \in A_0$, $\nu(x)=x^{\ell^{\nicefrac{f}{p}}}$. Then, there exists 
$\chi \in \irr(A)$ such that ${\chi_{A_0}}^{\nu} \neq \chi_{A_0}$, and such that either $\sigma$ fixes $\chi$ or $\chi^{\nu^m} = \chi^{\sigma}$ for some $m$ coprime with $p$.
\end{lemma}

\begin{proof}
Let $k = \ell^{\nicefrac{f}{p}}$, then $\abs{A_0}=k^p - 1$ and, since $p$ is odd, by Zsigmondy's Theorem we have that there exists a prime number $r$ which divides $\abs{A_0}$ and does not divide $k - 1$. Moreover, since $k^p \equiv_p k$, we have that $r \neq p$.

Consider now $\irr(A_0)$ and notice that it is an abelian group of order $k^p - 1$ and that, for any $\psi \in \irr(A_0)$, $\psi^{\nu} = \psi^k$. Now, fix $\theta \in \irr(A_0)$ of order $r$. Consider $V = \groupgen{\theta} \leq \irr(A_0)$ and notice that both $\nu$ and $\sigma$ act on $V$. Moreover, if $\psi \in V$ is fixed by $\nu$, then
$$ \psi^k = \psi \Longrightarrow \psi^{k-1} = 1 \Longrightarrow o(\psi) \mid k-1 $$
and it follows that $\psi = 1_{A_0}$. Suppose now that $\sigma$ does not fix $\theta$. Since, $\nu$ has order $p$ on $V$ (and on $A_0$ as well), $P=\groupgen{\nu}$ is a $p$-group and we can apply Lemma~\ref{lemma:action_on_minimal} to prove that $\theta$ is fixed by $\nu^m \times \sigma^{\shortminus 1} \in P \times \groupgen{\sigma}$, for some $m$ coprime with $p$.

Let $\rho \in P \times \groupgen{\sigma}$ such that $\rho = 1 \times \sigma$ if $\sigma$ fixes $\theta$, and $\rho = \nu^m \times \sigma^{\shortminus 1}$ otherwise. In both cases, $\rho$ has order $p$ and it fixes $\theta$. Since $\nu$ does not fix $\theta$, and since $\theta$ extends to $A$, we only need to prove that at least one of such extensions is fixed by $\rho$.

Let $A_0 \leq H \leq A$ such that $\abs{H/A_0}$ is the $r'$-part of $\abs{A/A_0}$. Then, $\theta$ has a canonical extension $\hat{\theta} \in \irr(H)$, which is fixed by $\rho$. Finally, since $p \nmid \abs{A:H}$, there exists some $\chi \in \irr(A)$ lying over $\hat{\theta}$ and fixed by $\rho$.
\end{proof}

We now need to state a rather simple property of semisimple characters. For a definition of semisimple characters, see \cite[Definition~2.6.9]{Geck-Malle:Character_Theory}.

\begin{lemma}
\label{lemma:semisimple_char}
Let $G = \mathbf{G}^F$ be a finite group of Lie type, for some Steinberg map $F$, and let $T=\mathbf{T}^F$ for some $F$-stable maximal torus $\mathbf{T} \leq \mathbf{G}$. Let $\theta \in \irr(T)$ and let $\chi \in \irr(G)$ be a semisimple character lying over $\theta$. Then,
\begin{itemize}
\item[a)] if $\nu \in \op{Aut}(G)$ normalizes $T$, then $\chi^{\nu}$ is a semisimple character lying over $\theta^{\nu}$;
\item[b)] if $\sigma \in \gal(\mathbb{Q}_{\abs{G}} / \mathbb{Q})$, then $\chi^{\sigma}$ is a semisimple character lying over $\theta^{\sigma}$.
\end{itemize}
\end{lemma}

\begin{proof}
It follows directly form the definition of semisimple character \cite[Definition~2.6.9]{Geck-Malle:Character_Theory}.
\end{proof}

If $S = \mathbf{S}^F$ is a finite group of Lie type, it is always possible to find a finite group $S \leq \tilde{S}$ such that $\tilde{S} = \mathbf{\tilde{S}}^{\tilde{F}}$, for some $\mathbf{S} \leq \mathbf{\tilde{S}}$ with connected centre and some Steinberg map $\tilde{F}$ such that $\tilde{F} = F$ on $\mathbf{S}$ (see \cite[Lemma~1.7.3]{Geck-Malle:Character_Theory}). With a little abuse of terminology, we may say that $\tilde{S}$ is a \emph{regular embedding} for $S$.

\begin{theorem}
\label{theorem:groups_Lie_type}
Let $S = \mathbf{S}^F$ be a finite group of Lie type, for some Steinberg map $F$, let $\tau$ be a field automorphism for $S$ of order $p^m$, for some prime number $p$, let $S < G \leq \op{Aut}(S)$ such that $G/S$ is a $p$-group and $\tau \in G$, and let $\sigma \in \gal(\mathbb{Q}_{\abs{G}} / \mathbb{Q})$ of order $p$.

Let $S \leq \tilde{S} \leq \op{Aut}(S)$ be a regular embedding for $S$. Suppose that $G/G\cap \tilde{S}$ is generated by the image of $\tau$ by the canonical projection, and that $p \nmid \abs{\tilde{S} : \tilde{S} \cap G}$. Then, there exists $\chi \in \irr(G)$ such that $\chi$ is $\sigma$-invariant and $p \mid \chi(1)$.
\end{theorem}

Notice that, if $S$ is simple, then the construction of \cite[Lemma~1.7.3]{Geck-Malle:Character_Theory} implies that $\ce{\tilde{S}}{S} = 1$. Thus, we can ask that $\tilde{S} \leq \op{Aut}(S)$. Moreover, as a consequence of \cite[Lemma~1.7.8]{Geck-Malle:Character_Theory}, $\abs{\tilde{S} : \tilde{S} \cap G}$ is deduced from the Classification of finite simple groups, for any simple group of Lie type $S$.

\begin{proof}
Suppose that $S$ is defined over a field of $q=\ell^f$ elements and let $\mathbf{T} \leq \mathbf{S}$ be an $F$-stable maximal torus such that $T = \mathbf{T}^F$ has a subgroup of order $q - 1$ (it is always possible to find such torus, see for instance the classification in \cite{Wilson:Finite_Simple_Groups}). By \cite[Lemma~1.7.7]{Geck-Malle:Character_Theory}, there exists an $\tilde{F}$-stable maximal torus $\mathbf{T} \leq \mathbf{\tilde{T}}$ of $\mathbf{\tilde{S}}$ such that, if $\tilde{T} = \mathbf{\tilde{T}}^{\tilde{F}}$, then $\tilde{S} = \tilde{T}S$ and $T = \tilde{T} \cap S$.

Let $\nu = \tau^{p^{m-1}}$, so that it has order $p$, and notice that $\nu(x) = x^{\ell^{\nicefrac{f}{p}}}$ for all $x \in T$. Moreover, since $p \nmid \abs{\tilde{S} : \tilde{S} \cap G}$, we can extend $\sigma$ to $\mathbb{Q}_{\abs{\tilde{S} : \tilde{S} \cap G} \abs{G}}$ and we still have an automorphism of order $p$ (we shall still call it $\sigma$). By Lemma~\ref{lemma:torus}, there exists $\theta \in \irr(\tilde{T})$ such that ${\theta_T}^{\nu} \neq \theta_T$ and $\theta$ is fixed either by $\sigma$ or by $\nu^m \times \sigma^{\shortminus 1}$, for some positive integer $m$ coprime with $p$.

Let $\rho \in P \times \groupgen{\sigma}$ such that $\theta^{\rho} = \theta$ and either $\rho = 1 \times \sigma$ or $\rho = \nu^m \times \sigma^{\shortminus 1}$, and let $\tilde{\psi} \in \irr(\tilde{S})$ be a semisimple character lying over $\theta$. By Lemma~\ref{lemma:semisimple_char}, $\tilde{\psi}^{\rho}$ is a semisimple character lying over $\theta^{\rho}=\theta$ and, since $\mathbf{\tilde{S}}$ has connected centre, by \cite[2.6.10]{Geck-Malle:Character_Theory} we have that $\tilde{\psi}^{\rho} = \tilde{\psi}$. Since $p \nmid \abs{\tilde{S} : \tilde{S} \cap G}$ (and since $\tilde{S} / S$ is abelian, because $\tilde{S} \leq \op{Aut}(S)$), we have that there exists at least one irreducible constituent $\psi$ of $\tilde{\psi}_{\tilde{S} \cap G}$ which is fixed by $\rho$. Moreover, it follows from \cite[Corollary~2.6.18]{Geck-Malle:Character_Theory} that all the irreducible constituents of $\psi_S$ are semisimple characters lying over $\theta_T \in \irr(T)$.

Now, notice that $\theta_T, {\theta_T}^{\nu} \in \irr(T)$ are not conjugated in $S$, since otherwise the Frattini argument would imply that $\nu \in I_G(\theta_T) S$ and, thus, $\nu \in I_G(\theta_T)$. Since the irreducible constituents of ${\psi^{\nu}}_S$ are semisimple characters lying over ${\theta_T}^{\nu}$, it follows from \cite[Theorem~2.6.2]{Geck-Malle:Character_Theory} that $({\psi^{\nu}}_S,\psi_S) = 0$ and, in particular, that $\psi$ is not fixes by $\nu$. Since, however, $\nu$ is contained in every proper subgroup of $\groupgen{\tau}$, we have that $I_G(\psi) = G \cap \tilde{S}$, and it finally follows from Lemma~\ref{lemma:normal_coprime} that $\chi = \psi^G \in \irr(G)$ is $\sigma$-invariant and $p \mid \chi(1)$.
\end{proof}

\begin{proof}[Proof of Theorem \ref{result:simple_groups}]
It follows from the Classification of finite simple groups that, in all the listed situations, $p$, $S$ and $G$ satisfy the conditions of Theorem~\ref{theorem:groups_Lie_type}. The existence of the group $\tilde{S}$ for Theorem~\ref{theorem:groups_Lie_type}, on the other hand, is guaranteed by \cite[Lemma~1.7.3]{Geck-Malle:Character_Theory}.
\end{proof}

\section{Proof of Theorem~\ref{result:p-solvable_groups} and Theorem~\ref{result:general_case}}

We are now ready to prove Theorem~\ref{result:p-solvable_groups} and Theorem~\ref{result:general_case}.

\begin{proof}[Proof of Theorem \ref{result:p-solvable_groups}]
We prove the theorem by induction on $\abs{G}$. Let $P$ be a Sylow $p$-subgroup of $G$ and suppose $P>1$, since otherwise the theorem follows trivially. We proceed by the following steps.

\textbf{Step 1.} We can assume $\rad{p}{G}=1$ and $\rup{p'}{G}=G$, otherwise the thesis follows by induction.

If $\rad{p}{G}>1$, then $G/\rad{p}{G}$ has a normal Sylow $p$-subgroup $P/\rad{p}{G}$ and, thus, $P$ is normal in $G$. Moreover, if $O=\rup{p'}{G}$ and $\psi \in \irr(O)$ is $\sigma$-invariant, then by Proposition~\ref{proposition:invariant_in_subgroup} there exists $\chi \in \irr(G \smid \psi)$ which is $\sigma$-invariant and, since $p \nmid \chi(1)$ by hypothesis, we have that $p \nmid \psi(1)$. It follows that $p$ does not divide the degree of any $\sigma$-invariant irreducible character of $O$. If $O < G$, it follows by induction that $P$ is a normal Sylow $p$-subgroup of $O$ and, therefore, $P$ is normal in $G$.

\textbf{Step 2.} Let $N$ be a minimal normal subgroup of $G$. Then, $N$ is a $p'$-group and, by induction, $PN/N \lhd G/N$. It follows that $PN=\rup{p'}{G}=G$.

Moreover, one may notice that $C=\ce{G}{N} \lhd G$ and, if $P_0 \in \syl{p}{C}$, then $P_0 = \rad{p}{C} \leq \rad{p}{G} = 1$. It follows that $\ce{P}{N} = P \cap C = 1$, and either $C=N$ and $N$ is abelian, or $C=1$.

\textbf{Step 3.} If $N$ is non-abelian, the thesis follows.

Suppose $N$ is non-abelian. Then, it is a direct product of non-abelian simple groups which are transitively permuted by $P$. Write $N=S_1 \times ... \times S_n$, where each $S_i$ is simple non-abelian. Write $S=S_1$ and let $H=\no{G}{S}$. Let $1_s \neq \delta \in \irr(S)$ be a rational character (see Theorem~\ref{theorem:rational_char}), let $\psi \in \irr(N)$ be defined as
$$ \psi = \delta \otimes 1_{S_2} \otimes ... \otimes 1_{S_n} $$
and notice that $\psi$ is $\sigma$-invariant, $I_G(\psi) \leq H$ and $(\abs{H/N},o(\psi)\psi(1)) = 1$. Thus, it follows from Lemma~\ref{lemma:normal_coprime} and from our hypothesis that $G=H=I_G(\psi)$, $N=S$ and, since $\ce{P}{N} = 1$, we have that $S < G \leq \op{Aut}(S)$. Since $p \nmid \abs{S}$, then $p$ must be odd and it follows from the Classification of finite simple groups that $S$ is a simple group of Lie type. Then, the thesis follows from Theorem~\ref{result:simple_groups}.

\textbf{Step 4.} The thesis follows also if $N$ is abelian.

Assume $N$ to be abelian, then it is an elementary abelian $q$-group for some $q \neq p$. Let $V=\irr(N)$, then $V$ is a finite vector space over a field of characteristic $q$; since $G=PN$ and $N$ is a minimal normal subgroup of $G$, $V$ is also an irreducible $P$-module, with $P$ acting on $V$ by conjugation. Since $\ce{P}{N}=1$ and $N$ is minimal, then also $\ce{N}{P}=1$ and $\ce{V}{P}=1_N$. Moreover, the action of $\sigma$ on $V$ commutes with the action of $P$. Then, by Lemma~\ref{lemma:action_on_minimal}, there exist $g \in P$ and $\lambda \in V=\irr(N)$ such that $\lambda^{\sigma}=\lambda^g$ and, by Lemma~\ref{lemma:normal_coprime}, there exists $\chi \in \irr(G)$ such that $\chi^{\sigma}=\chi$ and $\chi(1)=\abs{G:I_G(\lambda)}$. Since by hypothesis $p \nmid \chi(1)$, it follows that $G = I_G(\lambda)$ and, therefore, there exists some $1 \neq x \in N$ such that $\ce{G}{x} = G$. It follows that $x \in \ze{G} \cap N$ and, being $N$ a minimal normal subgroup of $G$, $N \leq \ze{G}$ and $G=P \times N$. Thus, the thesis follows also in this case.
\end{proof}

\begin{proof}[Proof of Theorem \ref{result:general_case}]
We prove the theorem by induction on $\abs{G}$. Let $N=\no{G}{P}$ and let $M/K$ be a chief factor of $G$, by induction we can assume that $K=1$ and, thus, that $M$ is a minimal normal subgroup of $G$. If $p \mid \abs{M}$, then $1 < P \cap M \leq N \cap M$ and the theorem is proved in this case. Thus, we can assume $M$ to be a $p'$-group.

Notice that, in this situation, $N \cap M = \ce{M}{P}$. Therefore, we need to prove that $C=\ce{M}{P} > 1$. We prove it both when $M$ is a non-abelian minimal normal subgroup and when $M$ is abelian.

\textbf{Case 1.} $M$ is non-abelian.

Suppose $M$ is a non-abelian $p'$-subgroup. Then, it is a direct product of non-abelian simple groups which are transitively permuted by $G$. Write $M=S_1 \times ... \times S_n$, where each $S_i$ is simple non-abelian, and write $S=S_1$. Let $1_S \neq \delta \in \irr(S)$ be a rational character (see Theorem~\ref{theorem:rational_char}). Let
$$\psi = \delta \otimes 1_{S_2} \otimes ... \otimes 1_{S_n} \in \irr(M)$$
and let $T=I_G(\psi)$. Let $P_0 = P \cap T$; by eventually replacing $\psi$ with some $G$-conjugate, we can assume that $P_0$ is a Sylow $p$-subgroup for $T$.

Let $\hat{\psi} \in \irr(P_0M \mid \psi)$ be the canonical extension of $\psi$ to $P_0M$, then $\hat{\psi}$ is $\sigma$-invariant and $\hat{\psi}(1)=\psi(1)$ divides the degree of every irreducible constituent of $\hat{\psi}^T$. Since $p \nmid \abs{T:P_0M}$, it follows from Proposition~\ref{proposition:invariant_in_subgroup} that there exists $\theta \in \irr(T)$ lying over $\hat{\psi}$, and thus lying also over $\psi$, such that $\theta^{\sigma}=\theta$. If $\chi = \theta^G \in \irr(G)$, then
$$ \chi^{\sigma} = (\theta^G)^{\sigma} = (\theta^{\sigma})^G = \theta^G = \chi $$
and, by hypothesis, $p \nmid \chi(1) = \abs{G:T}\theta(1)$. It follows that $p \nmid \abs{G:T}$ and, thus, $P \leq T$ and $\psi \in \irr_P(M)$, with $\irr_P(M)$ being the set of the $P$-invariant irreducible characters of $M$. By Glauberman correspondence, see \cite[Theorem~13.1]{Isaacs}, $\abs{\irr(C)} = \abs{\irr_P(M)} > 1$ and it follows that also $C > 1$.

\textbf{Case 2.} $M$ is abelian.

Assume $M$ to be abelian, then it is an elementary abelian $q$-group for some $q \neq p$. Let $V=\irr(M)$, $V$ has the structure of a finite vector space over a field of characteristic $q$ and it is also a $P$-module, with $P$ acting on $V$ by conjugation. If we suppose $C = \ce{M}{P}=1$, then also $\ce{V}{P}=1_M$. Moreover, the action of $\sigma$ on $V$ commutes with the action of $P$ and, for each $\xi \in V$, $\ker(\xi)=\ker(\xi^{\sigma})$ and, therefore, $\xi^{\sigma} \in \groupgen{\xi}$. By Lemma~\ref{lemma:action_on_space}, there exist $g \in P$ and $\lambda \in V$ such that $\lambda^{\sigma}=\lambda^g$. Moreover, we can choose $\lambda$ such that $\abs{P : I_P(\lambda)} = \min \{ \abs{P : I_P(\xi)} \mid 1_M \neq \xi \in \irr(M) \}$ and, in particular, it follows that $\abs{P : I_P(\lambda)} \mid \theta(1)$ for every $\theta \in \irr(PM)$ such that $M \nleq \ker(\theta)$.

Now, by Lemma~\ref{lemma:normal_coprime}, there exists $\eta \in \irr(PM)$ lying over $\lambda$ such that $\eta^{\sigma}=\eta$ and $\eta(1)=\abs{P : I_P(\lambda)}$. Let $\psi$ be an irreducible constituent of $\eta^G$, then $\psi_{PM} = n_1 \theta_1 + ... + n_r \theta_r$ for some $n_i \in \mathbb{N}$ and some $\theta_i \in \irr(PM)$ such that, for each $i \in \{1,...,r\}$, $M \nleq \ker(\theta_i)$, since $M \nleq \ker(\psi)$ and $M$ is normal in $G$. It follows from the previous paragraph that $\eta(1)$ divides the degree of every $\theta_i$, for each $i \in \{1,...,r\}$, and, thus, $\eta(1) \mid \psi(1)$.

Thus, $\eta(1)$ divides the degree of every irreducible constituent of $\eta^G$ and, since $p \nmid \abs{G:PM}$, it follows from Proposition~\ref{proposition:invariant_in_subgroup} that $\sigma$ fixes some irreducible constituent $\chi$ of $\eta^G$. By hypothesis, $p \nmid \chi(1)$ and, thus, $p \nmid \eta(1)=\abs{P : I_P(\lambda)}$. It follows that $1_M \neq \lambda \in \irr_P(M)$ and, again by Glauberman correspondence, $\abs{\irr(C)} = \abs{\irr_P(M)} > 1$ and, thus, $C > 1$.
\end{proof}

\begin{center}
\subsection*{Acknowledgements}
\end{center}

\noindent
The author thanks professor P. H. Tiep for the precious conversation while preparing the first version of this paper.

\noindent
The author is also grateful to professor M. A. Pellegrini for his help and support during the last phase of research on the topic here exposed.

\bigskip


\end{document}